\documentclass[leqno,a4paper]{article}
\usepackage[english]{babel}
\usepackage{amsmath}
\usepackage{amsthm}
\usepackage{amssymb}
\usepackage{enumerate}
\usepackage{xspace}
\usepackage{euscript}
\usepackage{graphicx}
\usepackage{graphics}
\usepackage{amscd}
\usepackage{epsfig}
\usepackage{tabularx}
\usepackage{mleftright}
\usepackage[usenames,dvipsnames,svgnames,table]{xcolor}
        \newtheorem{lemma}{Lemma}[section]
        
        \newtheorem{theorem}[lemma]{Theorem}
        \newtheorem{definition}{Definition}[section]

\numberwithin{equation}{section}

\title{\bf{EIT in a layered anisotropic medium}}
\author{Giovanni Alessandrini\thanks{Dipartimento di Matematica e Geoscienze, Universit\`{a} di Trieste, Italy. Email:alessang@units.it}\qquad{Maarten V. de Hoop\thanks{Departments of Computational and Applied Mathematics, Earth Science, Rice University, Houston, Texas, USA. Email:mdehoop@rice.edu}}\qquad\\
Romina Gaburro\thanks{Department of Mathematics and Statistics, Health Research Institute (HRI), University of Limerick, Ireland.  Email: romina.gaburro@ul.ie}\qquad Eva Sincich\thanks{Dipartimento di Matematica e Geoscienze, Universit\`{a} di Trieste, Italy. Email:esincich@units.it}}

\date{}

\begin{document}
\maketitle

\begin{abstract}
We consider the inverse problem in geophysics of imaging the
subsurface of the Earth in cases where a region below the surface is
known to be formed by strata of different materials and the depths and
thicknesses of the strata and the (possibly anisotropic) conductivity
of each of them need to be identified simultaneously. This problem is
treated as a special case of the inverse problem of determining a
family of nested inclusions in a medium $\Omega\subset\mathbb{R}^n$,
$n \geq 3$.
\end{abstract}

\section{Introduction}\label{sec1}
\setcounter{equation}{0}

We consider the inverse Calder\'on problem \cite{C}, also known as
Electrical Impedance Tomography (EIT) or, in geophysics, as Direct
Current (DC) method, of determining a matrix-valued conductivity
$\sigma(x)$ of a body $\Omega$ in which electrostatic equilibrium is
modelled by the elliptic equation
\begin{equation}\label{conductivity equation}
   \mbox{div}(\sigma\nabla u)=0\qquad\textnormal{in}\quad\Omega,
\end{equation}
where $u$ represents the electrostatic potential and the available
measurements are all possible pairs of current fluxes $\sigma\nabla
u\cdot\nu|_{\Sigma}$ and boundary voltages $u|_{\Sigma}$ collected on
a given open portion $\Sigma$ of $\partial\Omega$.

When $\sigma$ is isotropic, that is $\sigma =\gamma I$, $I$ denotes
the identity matrix and $\gamma$ is a scalar function on $\Omega$, a
vast literature is available and the theory has achieved a substantial
level of completeness, see, for instance \cite{U}. With the celebrated
counterexample by Tartar \cite{Ko-V1}, the general anisotropic problem
still poses several open issues. A principal line of investigation
concerning anisotropy in EIT has been of proving uniqueness modulo a
change of variables which fixes the boundary \cite{Le-U, Sy, N, La-U,
  La-U-T, Be, As-La-P}. In most applications, however, knowledge of
position and, hence, coordinates are important. In this direction,
certain, diverse results are available \cite{Ko-V1, Al, Al-G, Al-G1,
  G-Li, G-S, Li, I}. In \cite{Al-dH-G} a uniqueness result was
obtained when the unknown anisotropic conductivity is assumed to be
piecewise constant on a given domain partition, or segmentation, with
non-flat interfaces. Non-flatness shall be rigorously defined in 
Section \ref{Notation and definitions} where our other definitions are
given as well.  We recall that in \cite{Al-dH-G} we also specialized
Tartar's counterexample to the case of a half space and a constant
conductivity thus demonstrating that the non-flatness condition on
boundary and interfaces is necessary.

Here we address the more general problem when also the interfaces
defining the domain partition are unknown.  In this respect, in the
context of elastostatics, C\^arstea, Honda and Nakamura \cite{CHN}
obtained uniqueness from a local boundary map of a piecewise constant
anisotropic elasticity tensor where the partitioning is allowed to be
unknown, provided it is formed by subanalytic sets
\cite{BM}. Here, for EIT, we significantly relax the regularity
requirements on the domain partition on the one hand, but impose
stricter conditions on the configuration on the other hand.

More precisely, we treat
the case in which the interfaces are the (non-flat) $C^{1,\alpha}$ boundaries of a
nested family of subdomains $\Omega_k$, $\Omega_{k+1} \subset\subset
\Omega_k \subset\subset \Omega$, $k=1,\dots, K$. Within this setting,
assuming that the unknown conductivity $\sigma$ has the structure
\begin{equation}\label{sigma structure}
   \sigma = \sum_{k=1}^{K+1} \sigma_{k}
      \chi_{\left(\Omega_{k-1}\setminus\overline{\Omega}_k\right)},
\end{equation}
{where we understand that $\Omega_0 = \Omega$,
  $\Omega_{K+1} = \emptyset$ (so that the innermost layer consists of
  all of $\Omega_{K}$) and} $\sigma_k$ are positive definite constant
matrices satisfying the jump or visibility conditions
\begin{equation}\label{jump condition}
   \sigma_k \neq \sigma_{k+1},\quad
         \textnormal{for\:all}\quad k=1,\dots , K,
\end{equation}
we prove (Theorem~\ref{teorema principale} below) that $\sigma$ is
uniquely determined from the knowledge of the local
Neumann-to-Dirichlet (N-D) map
\begin{equation}\label{N-D local}
   \mathcal{N}^{\Sigma}_{\sigma} : \sigma\nabla u \cdot \nu|_{\Sigma}
                \longrightarrow u|_{\Sigma},
\end{equation}
for all solutions $u \in H^{1}(\Omega)$ to \eqref{conductivity
  equation}. Here $\Sigma$ is an open (non-flat) portion of
$\partial\Omega$.

From a geological perspective, the mentioned stratification arises
naturally in sedimentary basins, containing hydrocarbon reservoirs.
{Through the electrical conductivity, indeed, a geological image can be
obtained from boundary data. This is because the electrical
conductivity of Earth's materials varies over many orders of magnitude
while it depends upon many factors, including rock type, porosity,
connectivity of pores or permeability, nature of fluid, and metallic
content of a solid matrix. The representation of conductivity used in
this paper was motivated by the work of Loke, Acworth and Dahlin
\cite{Lo} and Farquharson \cite{F}. In the DC
inverse problem, the location of the boundaries or interfaces and the
conductivities are unknown, though the occurrence of stratification
might be inferred from independent or joint imaging of seismic data
\cite{H-O, Ga}. In reality, the
stratification will have a finite extent, that is, appear in some
cylindrical cut of the domain's interior or subsurface. We adapt our
analysis to this case, in a variation of Theorem \ref{teorema
  principale} in which we assume that $\sigma$ satisfies the layered
structure assumption only on a subdomain $C$ of which the ``top''
boundary, contained in $\partial\Omega$, is an appropriate
neighborhood of $\Sigma$ where measurements are collected. We then
uniquely determine $\sigma$ in $C$ only. In the absence of a
uniqueness result, this inverse problem has been extensively studied
in geophysics, mostly through experimenting with optimization and
sometimes motivated by a statistics framework \cite{M, E, Z, Gu}. 
Most recently, a regularization emphasizing sharp boundaries has been investigated by
Par\'{e} and Li \cite{P}; this strategy is closely aligned with
our analysis.}


\section{Main Result}\label{sec2}
\setcounter{equation}{0}


\subsection{Notation and definition}\label{Notation and definitions}

In several places in this manuscript it will be useful to single out one coordinate
direction. To this purpose, the following notations for
points $x\in \mathbb{R}^n$ will be adopted. For $n\geq 3$,
a point $x\in \mathbb{R}^n$ will be denoted by
$x=(x',x_n)$, where $x'\in\mathbb{R}^{n-1}$ and $x_n\in\mathbb{R}$.
Moreover, given a point $x\in \mathbb{R}^n$,
we shall denote with $B_r(x), B_r'(x)$ the open balls in
$\mathbb{R}^{n}, \mathbb{R}^{n-1}$ respectively centred at $x$ with radius $r$
and by $Q_r(x)$ the cylinder $B_r'(x')\times(x_n-r,x_n+r)$. We shall denote
$\mathbb{R}^n_+=\{(x',x_n)\in \mathbb{R}^n| x_n>0 \}$, $B^+_r=B_r\cap\mathbb{R}^n_+$, where we understand $B_r=B_r(0)$ and $Q_r=Q_r(0)$.


We shall assume throughout that $\Omega$ is a bounded domain with Lipschitz boundary, see e.g. \cite[4.9]{A-F}.



\begin{definition}\label{def C1 alpha boundary}
Let $\Omega$ be a domain in $\mathbb R^n$. Given $\alpha$,
$\alpha\in(0,1)$, we say that a portion $\Sigma$ of
$\partial\Omega$ is of class $C^{1,\alpha}$  if for any $P\in\Sigma$ there exists a rigid transformation of
$\mathbb R^n$ under which we have $P=0$ and
$$\Omega\cap Q_{r_0}=\{x\in Q_{r_0}\,|\,x_n>\varphi(x')\},$$
where $\varphi$ is a $C^{1,\alpha}$ function on $B'_{r_0}$
satisfying
\[\varphi(0)=|\nabla_{x'}\varphi(0)|=0.\]




\end{definition}

\begin{definition}\label{def C1 alpha non flat}
Given $\Sigma$ as above, we shall say that such a portion of a surface is non-flat (and equivalently the function $\varphi$) at a point $P\in\Sigma$ if, considering the  reference system and the function $\varphi$ as above, we have that $\varphi$ is not identically zero in any open neighborhood of $P=0$.
\end{definition}

\begin{definition}\label{Omega non flat}
We shall say that a whole boundary $\partial\Omega$ is non-flat if for each $P\in\partial\Omega$ there exists an open portion $\Sigma$ of $\partial\Omega$
such that $P\in\Sigma$, $\Sigma$ is of class $C^{1,\alpha}$ and it is non-flat at $P$.
\end{definition}

\subsubsection*{The Neumann-to-Dirichlet map.}\label{N-to-D}

We denote by $Sym_n$ the class of
$n\times n$ symmetric real valued matrices.  Let $\Omega$ be a domain
in $\mathbb{R}^n$ with Lipschitz boundary $\partial\Omega$ and assume
that $\sigma\in L^{\infty}(\Omega\:,Sym_{n})$ satisfies the
ellipticity condition

\begin{eqnarray}\label{ellitticita'sigma}
\lambda^{-1}\vert\xi\vert^{2}\leq{\sigma}(x)\xi\cdot\xi\leq\lambda\vert\xi\vert^{2},
& &for\:almost\:every\:x\in\Omega,\nonumber\\
& &for\:every\:\xi\in\mathbb{R}^{n}.
\end{eqnarray}

We shall also denote by $\langle\cdot,\cdot\rangle$ the
$L^{2}(\partial\Omega)$-pairing between
$H^{\frac{1}{2}}(\partial\Omega)$ and its dual
$H^{-\frac{1}{2}}(\partial\Omega)$.

We consider the following function spaces

\begin{equation*}
_{0}H^{\frac{1}{2}}(\partial \Omega)=\left\{f\in
H^{\frac{1}{2}}(\partial \Omega)\vert\:\int_{\partial\Omega}f\:
=0\right\},
\end{equation*}
\begin{equation*}
_{0}H^{-\frac{1}{2}}(\partial \Omega)=\left\{\psi\in
H^{-\frac{1}{2}}(\partial \Omega)\vert\:\langle\psi,\:1\rangle=0
\right\}.
\end{equation*}

We define the global Neumann-to-Dirichlet map as follows.

\begin{definition}\label{definition N-D}
The Neumann-to-Dirichlet (N-D) map associated with $\sigma$,

\[\mathcal{N}_{\sigma}:\  _{0}H^{-\frac{1}{2}}(\partial \Omega)\longrightarrow \:_{0}H^{\frac{1}{2}}(\partial \Omega)\]

is given by the selfadjoint operator satisfying

\begin{equation}\label{ND globale}
\langle\psi,\:\mathcal{N}_{\sigma}\psi\rangle\:=\:\int_{\:\Omega} \sigma(x)
\nabla{u}(x)\cdot\nabla{u}(x)\:dx,
\end{equation}
for every $\psi\in\: _{0}H^{-\frac{1}{2}}(\partial \Omega)$, where
$u\in{H}^{1}(\Omega)$ is the weak solution to the Neumann problem

\begin{equation}\label{N bvp}
\left\{ \begin{array}{lll}\displaystyle\textnormal{div}(\sigma\nabla u)=0, &
\textrm{$\textnormal{in}\quad\Omega$},\\
\displaystyle\sigma\nabla u\cdot\nu\vert_{\partial\Omega}=\psi, &
\textrm{$\textnormal{on}\quad{\partial\Omega}$},\\
\displaystyle\int_{\partial\Omega}u\: =0.
\end{array} \right.
\end{equation}

\end{definition}
{Note that from \eqref{ND globale} the bilinear form
\begin{equation*}\label{ND b}
\langle\varphi,\:\mathcal{N}_{\sigma}\psi\rangle, \:  \varphi, \psi \in\: _{0}H^{-\frac{1}{2}}(\partial \Omega) ,
\end{equation*}
can be defined by polarization in a 	straightforward fashion.}
Given $\sigma^{(i)}\in L^{\infty}(\Omega\:,Sym_{n})$, satisfying \eqref{ellitticita'sigma}, for $i=1,2$, the following identity can be recovered from Alessandrini's identity (see \cite[(b), p. 253]{Al}) {and the (obvious) equality
\[\mathcal{N}_{\sigma^{(1)}}^{-1}-\mathcal{N}_{\sigma^{(2)}}^{-1}=\mathcal{N}_{\sigma^{(1)}}^{-1}\left(\mathcal{N}_{\sigma^{(2)}} - \mathcal{N}_{\sigma^{(1)}}\right)\mathcal{N}_{\sigma^{(2)}}^{-1} ,
\]}
that is,
\begin{equation}\label{Alessandrini identity N-D}
\langle\sigma^{(1)}\nabla u_1\cdot\nu,\left(\mathcal{N}_{\sigma^{(2)}} - \mathcal{N}_{\sigma^{(1)}}\right)\sigma^{(2)}\nabla u_2\cdot\nu\rangle = \int_{\Omega} \left(\sigma^{(1)}(x) - \sigma^{(2)}(x)\right)\nabla u_1(x)\cdot\nabla u_2(x),
\end{equation}

for any $u_i\in H^{1}(\Omega)$ being the weak solution to

\begin{equation}\label{conductivity equations}
\textnormal{div}(\sigma^{(i)}(x)\nabla u_i(x))=0,\qquad\textnormal{in}\quad\Omega,
\end{equation}

for $i=1,2$.\\

Now we introduce the local version of the N-D map. Let  $\Sigma$ be an open portion of $\partial\Omega$ and let
$\Delta=\partial\Omega\setminus\overline\Sigma$. We introduce the subspace of $H^{\frac{1}{2}}(\partial \Omega)$,

\[H^{\frac{1}{2}}_{co}(\Delta)=\left\{f\in H^{\frac{1}{2}}(\partial \Omega)\:|\: \mbox{supp}(f)\subset\Delta\right\}.\]

We denote by $H^{\frac{1}{2}}_{00}(\Delta)$  the closure in $H^{\frac{1}{2}}(\partial\Omega)$ of the space
$H^{\frac{1}{2}}_{co}(\Delta)$ and we introduce

\begin{equation}
_{0}H^{-\frac{1}{2}}(\Sigma)=\left\{\psi\in \:
_{0}H^{-\frac{1}{2}}(\partial\Omega)\vert\:\langle\psi,\:f\rangle=0,\quad\textnormal{for\:any}\:f\in
H^{\frac{1}{2}}_{00}(\Delta)\right\},
\end{equation}
that is, the space of distributions $\psi \in
H^{-\frac{1}{2}}(\partial\Omega)$ which are supported in
$\overline\Sigma$ and have zero average on $\partial\Omega$. The local
N-D map is then defined as follows.

\begin{definition}
The local Neumann-to-Dirichlet map associated with $\sigma$,
$\Sigma$ is the operator $\mathcal{N}_{\sigma}^{\Sigma}:\:
_{0}H^{-\frac{1}{2}}(\Sigma)\longrightarrow
\big(_{0}H^{-\frac{1}{2}}(\Sigma)\big)^{\ast}\subset{_{0}H}^{\frac{1}{2}}(\partial\Omega)$
given  by
\begin{equation}
{\langle \varphi,\;\mathcal{N}_{\sigma}^{\Sigma}\psi\rangle=
\langle \varphi,\;\mathcal{N}_{\sigma}\psi\rangle,}
\end{equation}

for every $\varphi, \psi\in\:_{0}H^{-\frac{1}{2}}(\Sigma)$.
\end{definition}

Given $\sigma^{(i)}\in L^{\infty}(\Omega\:,Sym_{n})$, satisfying \eqref{ellitticita'sigma}, for $i=1,2$, we also recover from \eqref{Alessandrini identity N-D}

\begin{equation}\label{Alessandrini identity local N-D}
\left<\psi_1,\left(\mathcal{N}_{\sigma^{(2)}}^{\Sigma} - \mathcal{N}_{\sigma^{(1)}}^{\Sigma}\right)\psi_2\right> = \int_{\Omega} \left(\sigma^{(1)}(x) - \sigma^{(2)}(x)\right)\nabla u_1(x)\cdot\nabla u_2(x),
\end{equation}

for any $\psi_i\in\: _{0}H^{-\frac{1}{2}}(\Sigma)$, for $i=1,2$ and $u_i\in H^{1}(\Omega)$ being the unique weak solution to the Neumann problem

\begin{equation}
\left\{ \begin{array}{lll}\displaystyle\textnormal{div}(\sigma^{(i)}\nabla u_i)=0, &
\textrm{$\textnormal{in}\quad\Omega$},\\
\displaystyle\sigma^{(i)}\nabla u_i\cdot\nu\vert_{\partial\Omega}=\psi_i, &
\textrm{$\textnormal{on}\quad{\partial\Omega}$},\\
\displaystyle\int_{\partial\Omega}u_i\: =0.
\end{array} \right.
\end{equation}



\subsection{The a-priori assumptions}\label{subsection assumptions}


The assumptions pertaining to the domain partition are
\begin{enumerate}

\item $\Omega\subset\mathbb{R}^n$ is a bounded domain, with $n\geq 3$.

\item $\partial\Omega$ is of Lipschitz class.

\item We fix {a connected} open non-empty subset $\Sigma$ of $\partial\Omega$ (where the measurements in terms of the local N-D map are taken) and assume there exists $\alpha$, $\alpha\in(0,1)$ such that $\Sigma$ is $C^{1,\alpha}$ and non-flat.

More specifically we assume that there exists $P_0\in\Sigma$ and a rigid transformation of coordinates under which we have $P_0=0$ and

\begin{eqnarray}
\Sigma\cap{Q}_{r_{0}/3} &=&\{x\in
Q_{r_0/3}|x_n=\varphi_0(x')\}\nonumber\\
\left(\mathbb{R}^n\setminus\Omega\right)\cap {Q}_{r_{0}/3} &=&\{x\in
Q_{r_0/3}|x_n<\varphi_0(x')\}\nonumber\\
\Omega\cap {Q}_{r_{0}/3} &=&\{x\in
Q_{r_0/3}|x_n>\varphi_0(x')\},
\end{eqnarray}

where $\varphi_0$ is a non-flat $C^{1,\alpha}$ function on $B'_{r_o/3}$ satisfying

\[\varphi_0(0)=|\nabla\varphi_0(0)|=0.\]

\item Let {$K$ be a positive integer and let} $\Omega_0,\:\Omega_1,\dots , \Omega_K$ be nested domains

\[\Omega_K\subset\subset\Omega_{K-1}\subset\subset\dots\Omega_0 =\Omega.\]

For $k=1,\dots , K$ we denote

\begin{equation}\label{definition Dk}
D_k =\Omega_{k-1}\setminus\overline{\Omega}_k,
\end{equation}

where for $k=K+1$ we set

\[D_{K+1}=\Omega_K\]

{and we assume that all $D_k$ are connected}.

\item We assume that $\partial\Omega_k$ is $C^{1,\alpha}$ and it is non-flat according to Definition \ref{Omega non flat}, for every $k=1,\dots , K$.





\end{enumerate}

We then assume that the conductivity $\sigma\in
L^{\infty}(\Omega,\:Sym_n)$, satisfies {the uniform
  ellipticity condition} \eqref{ellitticita'sigma} and is of type
\begin{equation}\label{a priori info su sigma}
\sigma(x)=\sum_{k=1}^{K+1}\sigma_{k}\chi_{D_k}(x),\qquad
x\in\Omega,
\end{equation}
where the $\sigma_{k}$ are positive definite constant matrices, for $k=1,\dots ,K+1$ .


\subsection{Global uniqueness}

Our main result is stated below.

\begin{theorem}\label{teorema principale}
Let $K_i\in\mathbb{N}\setminus\{0\}$ and $\Omega$, $\Sigma$, $\Omega^{(i)}_k$, $k=0,\dots , K_i$,  $D^{(i)}_k$, $k=1,\dots , K_i+1$, for $i=1,2$ satisfy assumptions $1.-5.$ of subsection \ref{subsection assumptions}. If $\sigma^{(i)}$, $i=1,2$ are two conductivities of type

\begin{equation}\label{conduttivita anisotrope}
\sigma^{(i)}(x)=\sum_{k=1}^{K_i+1}\sigma_{k}^{(i)}\chi_{D^{(i)}_k}(x),\qquad
x\in\Omega,\:i=1,2,
\end{equation}

where $\sigma_{k}^{(i)}\in Sym_n$ are positive definite constant matrices satisfying

\begin{equation}\label{different matrices}
\sigma_{k}^{(i)}\neq\sigma_{k+1}^{(i)}\qquad k=1,\dots , K_i
\end{equation}

and the uniform ellipticity condition \eqref{ellitticita'sigma}, for $k=1,\dots , K_i+1$ and

\[\mathcal{N}^{\Sigma}_{\sigma^{(1)}}=\mathcal{N}^{\Sigma}_{\sigma^{(2)}},\]

then

\begin{equation}\label{equality number of domains}
K_1 = K_2:=K,
\end{equation}

\begin{equation}\label{uniqueness domains and conductivities}
\Omega^{(1)}_k = \Omega^{(2)}_k \quad\mbox{and}\quad\sigma^{(1)}_{k+1}=\sigma^{(2)}_{k+1},\qquad\textnormal{for\:any}\quad k=0,\dots , K.
\end{equation}

\end{theorem}


\section{Proof of the main result}

\begin{proof}[Proof of Theorem \ref{teorema principale}]
We assume without loss of generality that $K_1=\mbox{min}\left\{K_1,\:K_2\right\}$. First, we prove \eqref{uniqueness domains and conductivities} for $k=0, \dots, K_1$. We proceed by induction on $k$, $0\leq k\leq K_1$. For the case $k=0$, $\Omega^{(1)}_0 = \Omega = \Omega^{(2)}_0$ trivially holds true. By rephrasing the arguments used in \cite[Theorem 2.1]{Al-dH-G} we obtain that the equality of the maps

\[\mathcal{N}^{\Sigma}_{\sigma^{(1)}}=\mathcal{N}^{\Sigma}_{\sigma^{(2)}}\]

implies that, denoting by $E_1$ the connected component of  $\Omega\setminus\overline{\left(\Omega^{(1)}_1\cup\Omega^{(2)}_1\right)}$ such that $\Sigma\subset\partial E_1$, we have

\begin{equation}\label{step 0}
\sigma^{(1)}_1=\sigma^{(2)}_1\qquad\textnormal{in}\:E_1.
\end{equation}
In fact  \eqref{step 0} is obtained as follows.\\

- The knowledge of {$\mathcal{N}^{\Sigma}_{\sigma}$} enables {us} to determine the tangential asymptotics near the singularity of the Neumann kernel $N_{\sigma}(\cdot,y)$ for each $y\in \Sigma$ \cite[Lemma 3.8]{Al-dH-G}. {Here $N_{\sigma}(\cdot,y)$ is defined as the distributional solution of the following boundary value problem
\begin{displaymath}
\left\{ \begin{array}{ll}
\displaystyle\textnormal{div}(\sigma\nabla N_{\sigma}(\cdot,y))=0, & \textnormal{in}\quad\Omega\\
\sigma\nabla N_{\sigma}(\cdot, y)\cdot\nu= \delta(\cdot -y)-\frac{1}{\vert\partial\Omega\vert},
& \textnormal{on}\quad{\partial\Omega}.
\end{array} \right.
\end{displaymath}}
\\

- The tangential asymptotics of $N_{\sigma}(\cdot, y)$ allows us to identify the tangential $(n-1)\times(n-1)$ submatrices $g_{n-1}(y)$ of the metric 

\begin{equation*}\label{g}
g=\left(\det\sigma\right)^{\frac{1}{n-2}}\sigma^{-1}
\end{equation*} 

associated to the elliptic operator $\mbox{div}(\sigma\nabla\cdot)$ \cite[Lemma 3.5]{Al-dH-G}.\\

- The non-flatness assumptions of $\Sigma$ permits us to find enough independent tangent planes so to determine all of $g$ (hence $\sigma$) provided $\sigma$ is locally constant \cite[Lemma 3.6]{Al-dH-G}.\\

{Next we prove the induction step.} Let $1\leq k\leq K_1$. We assume that for every $j=0,\dots , k-1$

\begin{equation}\label{inductive step}
\Omega^{(1)}_j = \Omega^{(2)}_j\quad\textnormal{and}\quad \sigma^{(1)}_{j+1}=\sigma^{(2)}_{j+1}
\end{equation}

and suppose by contradiction that

\begin{equation}\label{assumption by contradiction}
\partial\Omega^{(1)}_{k}\setminus \overline{\Omega^{(2)}_{k}}\neq\emptyset
\end{equation}

(the symmetric case $\partial\Omega^{(2)}_{k}\setminus\overline{\Omega^{(1)}_{k}}\neq\emptyset$ being equivalent). We denote by $E_k$ the connected component of $\Omega\setminus\overline{\Omega^{(1)}_k\cup\Omega^{(2)}_k}$ such that $\Sigma\subset\partial{E}_k$. 
Let us fix an open portion of ${(\partial\Omega^{(1)}_{k}\setminus \overline{\Omega^{(2)}_{k}})} \cap \partial{E}_k$, which we denote by $\Sigma_k$
and let us select a subdomain $\mathcal{E}_k\subset E_k$ such that $\mathcal{E}_k$ and $\mathcal{F}_k=\Omega\setminus\overline{\mathcal{E}_k}$ have both Lipschitz boundary and such that

\begin{equation*}
\Sigma\cup\Sigma_k\subset\partial\mathcal{E}_k.
\end{equation*}

We note that  $\sigma^{(1)}=\sigma^{(2)}$ in $\mathcal{E}_k$. Let us denote by 

\[\mathcal{N}^{\Sigma_{k}}_{\sigma^{(i)}}\]

the local N-D maps for $\sigma^{(i)}$ in $\mathcal{F}_k$. Then \cite[Claim 4.1]{Al-dH-G} implies that

\begin{equation}\label{equality N 1}
\mathcal{N}^{\Sigma_{k}}_{\sigma^{(1)}} = \mathcal{N}^{\Sigma_{k}}_{\sigma^{(2)}}.
\end{equation}
{Note that the set $D$ appearing in \cite[Claim 4.1]{Al-dH-G} needs to be replaced by $\mathcal{E}_k$.}

Let us fix  $y_{k}\in\Sigma_{k}$ and a neighborhood $U_{k}$ of $y_{k}$ in $\mathcal{F}_k$ such that $U_{k}\cap\overline{\Omega^{(2)}_{k}}=\emptyset$. We can choose $U_k$ small enough so that $\sigma^{(1)}$ and $\sigma^{(2)}$ are both constant in $U_{k}$. Using once more \cite[Lemma 3.6]{Al-dH-G} we obtain

\begin{equation}\label{equality sigma 1 semifinal}
\sigma^{(1)} _{k+1}= \sigma^{(2)}_{k},
\end{equation}

which, combined with \eqref{inductive step}, implies that

\begin{equation}\label{equality sigma 1 final}
\sigma^{(1)} _{k+1}= \sigma^{(1)}_{k}.
\end{equation}

Hence by \eqref{different matrices} we have reached a contradiction with the assumption \eqref{assumption by contradiction}, and therefore

{\begin{equation}\label{proof main result part 1}
\Omega^{(1)}_{k}=\Omega^{(2)}_{k}.
\end{equation}
Once we know that such domains coincide, \eqref{equality N 1} in combination with \cite[Lemma 3.6]{Al-dH-G} implies

\begin{equation}\label{equal sigma}
\sigma^{(1)}_{k+1}=\sigma^{(2)}_{k+1} ;
\end{equation}
thus, the induction step is proven and \eqref{uniqueness domains and conductivities} holds true for $k=0,\dots , K_1$}. In particular, this implies that

\begin{equation}\label{ug}
\Omega^{(1)}_{K_1}=\Omega^{(2)}_{K_1}:=\Omega_{K_1}\quad\mbox{and}\quad\sigma^{(1)}_{K_1 +1}=\sigma^{(2)}_{K_1 +1}\quad\mbox{on}\quad D^{(2)}_{K_1+1}.
\end{equation}

Next, we show that $K_1 = K_2 :=K$. Suppose, on the contrary, $K_2>K_1$ and denote

\begin{equation}\label{inclusions tilde}
\widetilde{\Omega}^{(1)}_k = \Omega^{(2)}_k,\qquad\mbox{for}\:k=K_1 +1,\dots , K_2,
\end{equation}

\begin{equation}\label{inclusions tilde 2}
\widetilde{D}^{(1)}_k = D^{(2)}_k,\qquad\mbox{for}\:k=K_1 +1,\dots , K_2
\end{equation}

and



\begin{equation}\label{conductivity 1 on inclusions tilde}
\widetilde{\sigma}^{(1)}_k=\sigma^{(1)}_{K_1 +1},\qquad\mbox{for}\quad k=K_{1}+1,\dots , K_2.
\end{equation}



Let $\Sigma_{K_1+1}$ be a non-empty portion of $\partial\widetilde{\Omega}^{(1)}_{K_1+1}=\partial\Omega^{(2)}_{K_1+1}$. By the same argument adopted above, we have that

\[\sigma^{(1)}=\sigma^{(2)}\qquad\mbox{on}\quad\Omega\setminus\overline{\widetilde{\Omega}^{(1)}_{K_1+1}}=\Omega\setminus\overline{\Omega^{(2)}_{K_1+1}},\]

which combined with \cite[Claim 4.1]{Al-dH-G} leads to

\begin{equation*}
\mathcal{N}^{\Sigma_{K_1 +1}}_{\sigma^{(1)}} = \mathcal{N}^{\Sigma_{K_1 +1}}_{\sigma^{(2)}},
\end{equation*}

whence

\begin{equation}\label{ug2}
\sigma^{(1)}_{K_1 +2}=\sigma^{(2)}_{K_1 +2}.
\end{equation}

Moreover, by \eqref{conductivity 1 on inclusions tilde} we have that

\[\sigma^{(1)}_{K_1 +1}=\sigma^{(1)}_{K_1 +2},\] \

which combined with \eqref{ug} and \eqref{ug2}, implies that  \[\sigma^{(2)}_{K_1 +1}=\sigma^{(2)}_{K_1 +2}.\]

The latter contradicts \eqref{different matrices}, therefore $K_1=K_2$ which concludes the proof.

\end{proof}


We conclude with presenting a variation of the result obtained in
Theorem \ref{teorema principale} which, we believe, is of interest in
the context of imaging materials with a structure that is layered
locally only. Let $\Sigma$ and $\varphi$ be the portion (where the
measurements are collected) and the function respectively introduced
in Definitions \ref{def C1 alpha boundary}, \ref{def C1 alpha non
  flat}.





We denote

\begin{equation}\label{C}
C=\left\{x\in\mathbb{R}^n\: |\: |x'|\leq R,\:\varphi\leq x_n\leq M\right\},
\end{equation}

for some {{positive numbers $R$ and $M$}}. Suppose $C\subset\Omega$ and also that

\begin{equation}\label{C boundary}
\partial C\cap\partial\Omega=\left\{x\in\mathbb{R}^n\: |\: |x'|\leq R,\quad x_n =\varphi(x') \right\}\supset\Sigma.
\end{equation}

Let $\varphi_1,\dots , \varphi_K:B'_R \longrightarrow\mathbb{R}$ be $C^{1,\alpha}$ functions, non-flat at every point as in Definitions \ref{def C1 alpha boundary}, \ref{def C1 alpha non flat}  which satisfy

\begin{equation}\label{phi layers}
\varphi(x')\equiv\varphi_0 (x')<\varphi_1 (x')<\dots <\varphi_K (x')<M,\qquad\textnormal{for\:all}\quad x'\in\overline{B_R'}.
\end{equation}

For ${{k}}=1,\dots , K$ denote

\begin{equation}\label{D layers}
{{D_k}}=\left\{x\in C\: | \:{{\varphi_{k-1}}}(x')< x_n <{{\varphi_k}}(x')\right\}
\end{equation}

and assume that $\sigma\in L^{\infty}(\Omega,\: Sym_n)$ satisfies { the uniform ellipticity condition} \eqref{ellitticita'sigma} and

\begin{equation}\label{a priori info su sigma 2}
\sigma(x)={{\sum_{k=1}^{K}}}{{\sigma_{k}\chi_{D_k}}}(x),\qquad
x\in C,
\end{equation}

where each {{$\sigma_{k}$}} is a positive definite constant {matrix} and 

\begin{equation}\label{sigmas}
{{\sigma_k\neq\sigma_{k+1},\quad\textnormal{for\: all}\quad k=1,\dots , K-1}}.
\end{equation}

With the above setting, we have the following uniqueness result
confined to a subdomain $C$ of $\Omega$

\begin{theorem}
$\mathcal{N}^{\Sigma}_{\sigma}$ uniquely determines $\sigma$ within $C$.
\end{theorem}

\begin{proof}
The proof follows the same line of the proof of Theorem \ref{teorema principale}.  
Let $\sigma^{(1)},\sigma^{(2)}$ satisfy the above structure conditions, {that is,

\begin{equation}\label{conduttivita anisotrope layers}
\sigma^{(i)}(x)=\sum_{k=1}^{K_i}\sigma_{k}^{(i)}\chi_{D^{(i)}_k}(x),\qquad
x\in\Omega,\:i=1,2,
\end{equation}

where $\sigma_{k}^{(i)}\in Sym_n$ are constant matrices and he layers $D_k^{(i)}$ are described by the functions $\varphi_j^{(i)}, i=1,2$.} The fact that, within $C$, the various interfaces are graphs with respect to the same reference system, enables us to select the inner boundary portions $\Sigma_k$ in such a way that they are all contained in $C$. The sets $\mathcal E_k$ can be explicitly expressed as 
\begin{eqnarray*}
\{ x\in \mathbb{R}^n \ : \  |x'|<R \ ,  \ \varphi\le x_n\le \min \{{\varphi}_k^{(1)}, {\varphi}_k^{(2)} \}  \} \ .
\end{eqnarray*}

 Hence $\mathcal{N}^{\Sigma}_{\sigma^{(1)}}=\mathcal{N}^{\Sigma}_{\sigma^{(2)}}$ leads to $\sigma^{(1)}=\sigma^{(2)}$ within the set $C$.
 
\end{proof}

\section*{\normalsize{Acknowledgments}}
The research carried out by G. Alessandrini and E. Sincich for the preparation of this paper has been supported by FRA 2016 "Problemi inversi, dalla stabilit\`{a} alla ricostruzione" funded by Universit\`{a} degli Studi di Trieste. M.V de Hoop was partially supported by the Simons Foundation under the MATH $+$ X program, the National Science Foundation under grant DMS-1559587, and by the members of the Geo-Mathematical Group at Rice University. R. Gaburro
acknowledges the support of MACSI, the Mathematics Applications Consortium for Science and Industry (www.macsi.ul.ie), funded by the Science Foundation Ireland Investigator Award 12/IA/1683. E. Sincich has also been supported by Gruppo Nazionale per l' Analisi Matematica, la
Probabilit\`a e le loro Applicazioni (GNAMPA) by the grant ''Analisi di problemi inversi:  stabilit\`a e
ricostruzione'' .


\end{document}